\documentclass[a4paper,12pt]{article}
\usepackage{CJKutf8}
\usepackage{CJK}
\usepackage{amsfonts}
\usepackage{pifont}
\usepackage{bm}
\usepackage{latexsym,amsmath,amssymb,cite,amsthm}
\usepackage{color,eucal,enumerate,mathrsfs}
\usepackage[normalem]{ulem}
\usepackage{amsmath}
\usepackage[pagewise]{lineno}
\usepackage[colorlinks, linkcolor=red,  anchorcolor=blue,citecolor=blue]{hyperref}

\numberwithin{equation}{section}
\theoremstyle{plain}
\newtheorem{theorem}{Theorem}[section]
\newtheorem{corollary}[theorem]{Corollary}
\newtheorem{lemma}[theorem]{Lemma}
\newtheorem{proposition}[theorem]{Proposition}
\theoremstyle{definition}
\newtheorem{definition}[theorem]{Definition}
\theoremstyle{remark}
\newtheorem{assumption}[theorem]{Assumption}
\newtheorem{remark}[theorem]{Remark}

\newcommand{\lmt}[2]{\mathop{\lim}_{{#1} \rightarrow {#2}} }

\newcommand{\lip}[1]{{\mathrm{lip}}({#1})}

\newcommand{\lmts}[2]{\mathop{\overline{\lim}}_{{#1} \rightarrow {#2}} }

\newcommand{\tr}{{\rm{tr}}}
\renewcommand{\H}{{\mathrm{Hess}}}

\newcommand{\mm}{\mathfrak m}
\newcommand{\ms}{(X,\d,\mm)}

\newcommand{\cdkn}{{\rm CD}(K, N)}
\newcommand{\rcdkn}{{\rm RCD}(K, N)}
\newcommand{\rcd}{{\rm RCD}(K, \infty)}
\newcommand{\be}{{\rm BE}(K, \infty)}
\newcommand{\bekn}{{\rm BE}(K, N)}
\newcommand{\BE}{{\rm BE}}

\newcommand{\B}{\mathrm{B}}

\newcommand{\E}{\mathcal{E}}

\renewcommand{\P}{\mathrm{P}}
\newcommand{\N}{\mathbb{N}}

\newcommand{\V}{\mathbb{V}}
\newcommand{\R}{\mathbb{R}}

\newcommand{\llip}{{\mathrm{lip}}}

\newcommand{\supp}{\mathop{\rm supp}\nolimits}   
\newcommand{\Lip}{\mathop{\rm Lip}\nolimits}
\newcommand{\loc}{{\rm loc}}
\renewcommand{\d}{{\mathrm d}}

\newcommand{\D}{\mathrm{D}}

\newcommand{\dbdelta}{{\rm D}({\bf \Delta})}
\newcommand{\restr}[1]{\lower3pt\hbox{$|_{#1}$}}
\newcommand{\la}{{\langle}}
\newcommand{\ra}{{\rangle}}

\newcommand{\nchi}{{\raise.3ex\hbox{$\chi$}}}

\title{\large{\bf Curvature-dimension conditions  
under time change}
}
\begin{document}
\author{Bang-Xian Han\thanks{Technion-Israel Institute of Technology, hanbangxian@gmail.com.
}\and Karl-Theodor Sturm\thanks
{Universit\"{a}t Bonn,   sturm@uni-bonn.de }
}

\date{}
\maketitle

\begin{abstract}
We derive precise transformation formulas for synthetic lower Ricci bounds under time change. 
More precisely,  for local Dirichlet forms we study how
the curvature-dimension condition in the sense of Bakry-\'Emery will transform under time change. Similarly,  for metric measure spaces we
study how the curvature-dimension condition in the sense of Lott-Sturm-Villani will transform under time change.
\end{abstract}

\tableofcontents

\section{Introduction}
\paragraph{A.}
Bakry and \'Emery \cite{BE-D} formulated a powerful criterion for obtaining equilibration and regularity results for the Markov semigroups associated with local Dirichlet forms. Let us briefly recall their concept.
A Dirichlet form    $\E$, densely defined  on some $L^2(X,\mm)$, satisfies the    BE$(k,N)$ condition with some function $k \in L_\loc^\infty(X,\mm)$ and some  number $N\in [1,\infty]$   if 
\begin{equation}\label{eq1-intro}
\frac12 \int \Gamma(f) \Delta \varphi\,\d\mm-\int \Gamma(f, \Delta f) \varphi\,\d\mm \geq \int \Big (k\, \Gamma(f)+ \frac1N (\Delta f)^2 \Big )\varphi\,\d\mm.
\end{equation}
for all suitable functions $f$ and $\varphi\ge0$ on $X$. Here $\Delta$ denotes the generator associated with $\E$ and $\Gamma$ the carr\'e du champ operator. Estimate  \eqref{eq1-intro} can be regarded as an 
 abstract formulation of  Bochner's inequality on Riemannian manifolds. Thus, in this Eulerian approach to curvature-dimension conditions, $k(x)$ will be considered as  a synthetic lower bound for the ``Ricci curvature at $x\in X$'' and $N$  as an upper bound for the ``dimension'' of $X$.

From the very beginning of this theory, the transformation formula for the Bakry-\'Emery condition $\BE(k,N)$ under {\em drift transformation} played  a key role. Most importantly in the case $N=\infty$, this states that the Dirichlet form 
$$\E^*(u):=\int\Gamma(u)\,\d\mm^*\text{ on }L^2(X,\mm^*)\text{ with }\mm^*:=e^{-V}\,\mm$$
satisfies $\BE(k^*,\infty)$ with $k^*:=k+h_V$ where
$h_V(x):=\inf_f \frac1{\Gamma(f)}\big[\Gamma\big(\Gamma(V,f),f\big)-\frac12\Gamma\big(\Gamma(f),V\big)\big]$
denotes the lower bound for the Hessian of $V$ at $x\in X$ for any sufficiently regular function $V$ on $X$.

The goal of this paper now is to analyze the transformation property of the Bakry-\'Emery condition under {\em time change}. That is, we will pass from the original Dirichlet form $\E$ on $L^2(X,\mm)$ to a new one defined as
$$\E'(u):=\int \Gamma(u)\,\d\mm\text{ on }L^2(X,\mm')\text{ with }\mm':=e^{2w}\,\mm$$
for some $w\in L^\infty_\loc(X,\mm)$.
Our main result provides a Bakry-\'Emery condition for this transformed Dirichlet form provided the original Dirichlet form satisfies a Bakry-\'Emery condition with finite $N$.

\begin{theorem}\label{Thm A} Assume that  $\E$ satisfies  the ${\rm BE}(k,N)$  condition for some  $k \in L^\infty_\loc$ and some  $N\in [1,\infty)$, and that  $w\in  {\rm D}_\loc({ \bf \Delta}) \cap L_\loc^\infty$  with ${\bf \Delta} w={\bf \Delta}_{sing} w+{ \Delta}_{ac} w \, \mm$ and ${\bf \Delta}_{sing} w  \leq 0$. 
Then for any $N'\in (N, \infty]$ and $k' \in L_\loc^\infty$, the time-changed Dirichlet form  $\E'$ on $L^2(X,\mm')$ satisfies the ${\rm BE}(k', N')$ condition  provided
\begin{equation}
k'\le e^{-2w}\Big [k- \frac{(N-2)(N'-2)}{N'-N} \Gamma(w)- {\Delta}_{ac} w  \Big]. 
\end{equation}
%
\end{theorem}

\begin{corollary} If in addition $k'$ is bounded from below, say $k'\ge K'$ for some $K'\in \R$, then the time changed Dirichlet form $\E'$ and the associated heat semigroup $(\P'_t)_{t\ge0}$ satisfy the  following gradient estimate
\begin{equation}
 \Gamma'(P'_t f) +\frac{1-e^{-2K't}}{N'K'}  (\Delta' P'_t f)^2\leq e^{-2K't} P'_t \big (\Gamma'(f)\big ).
\end{equation}
\end{corollary}

\begin{remark}
Generator and carr\'e du champ operator of the time-changed Dirichlet form $\E'$ on $L^2(X,\mm')$ are given by
\[
\Delta'=e^{-2w} \Delta, ~~ \Gamma'=e^{-2w} \Gamma.
\]
Moreover, 
the associated Brownian motion $(\mathbb P'_x, \B'_t)$  (c.f. Chapter 6 \cite{FOT}) is given by $\mathbb P'_x=\mathbb P_x$ and
\begin{equation}\label{time-change:BM}
\B'_t=\B_{\tau_t},\quad \tau_t=\int_0^t e^{-2w(\B'_s)}ds,\qquad  \sigma_t=\int_0^t e^{2w (\B_s)}\, \d s,\quad
\B_t=\B'_{\sigma_t}.
\end{equation}
Note that heat semigroup $(P'_t)_{t\ge0}$ and Brownian motion $(\mathbb P'_x, \B'_t)$ are linked to each other by
$$P'_t f(x)
=\mathbb E'_x[f(\B'_{2t})].$$ 
\end{remark}

\paragraph{B.}
A different approach, the so-called Lagrangian approach, to synthetic lower Ricci bounds was proposed in the works of  Lott, Villani \cite{Lott-Villani09} and Sturm \cite{S-O1}. 
Here the objects under consideration are metric measure spaces. Such a space  $(X,\d,\mm)$ satisfies  the curvature-dimension condition CD$(K,\infty)$ -- meaning that its  Ricci curvature is bounded from below by $K$ --
if the Boltzmann entropy $\operatorname{Ent}(.,\mm)$ is weakly $K$-convex on the Wasserstein space ${\mathcal P}_2(X)$. More refined curvature-dimension conditions CD$(K,N)$ and $\text{CD}^*(K,N)$ with finite $N\in[1,\infty)$ were introduced in \cite{S-O2} and \cite{S-O1}. Combined with the requirement of Hilbertian energy functional,
this led to the conditions RCD$(K,N)$ and $\text{RCD}^*(K,N)$ \cite{AGS-M}, which fortunately turned out to be equivalent to each other \cite{Cavalletti-Milman16}.

Also from the very beginning of this theory, the transformation formula for the curvature-dimension conditions CD$(K,N)$, $\text{CD}^*(K,N)$ RCD$(K,N)$ under {\em drift transformation} played  a key role. Most easily formulated in the case $N=\infty$, it states that the condition CD$(K,\infty)$ for  a given metric measure space $(X,\d,\mm)$ and the $L$-convexity of $V$ on $X$ imply the condition CD$(K+L,\infty)$ for  the transformed  metric measure space $(X,\d,e^{-V}\mm)$. The same holds with RCD in the place of CD.

Subject of the investigations in this paper is the {\em time-changed metric measure space} $(X,\d',\mm')$ where $\mm'=e^{2w}\mm$ for some $w\in L^\infty_\loc(X,\mm)$ and 
\begin{equation*}
\d'(x, y):=\sup \big \{\phi(x)-\phi(y): \phi \in {\rm D}_\loc(\E) \cap C(X),  |\D\phi| \leq e^w \ \mm\text{-a.e. in}~ X \big \}
\end{equation*}
for $x, y \in X$.
Assuming that $w$ is continuous $\mm$-a.e.~on $X$ this allows for a dual representation as
\begin{equation*}
\d'(x, y)=\inf\Big\{\int_0^1 e^{\bar w (\gamma_s)}\, |\dot\gamma_s|\,\d s:  \gamma \in {\rm AC}([0,1], X), \gamma_0=x,\gamma_1=y\Big\}
\end{equation*}
where 
$\bar w(x):=\limsup_{y\to x}w(y)$ denotes the  upper semicontinuous envelope of  $w$. Our main result provides the transformation formula for the curvature-dimension condition under time change.

\begin{theorem}\label{Thm B}
Let  $\ms$ be a $\rcdkn$ space and let $w\in  {\rm D}_\loc({\bf \Delta})\cap  L^\infty_\loc(X)$ be continuous  $\mm$-a.e.~with ${\bf \Delta} w={\bf \Delta}_{sing} w+{ \Delta}_{ac} w \, \mm$ and ${\bf \Delta}_{sing} w  \leq 0$.  
Then the time-changed  metric measure space $(X, \d', \mm')$ satisfies the
 ${\rm RCD}(K', N')$ condition
 for any  $N'\in (N, +\infty]$ and $K'\in\R$ such that
\[K'\le
 e^{-2w}\Big [K- \frac{(N-2)(N'-2)}{N'-N} |\D w|^2- { \Delta}_{ac} w \Big].
\]
\end{theorem}

Theorem \ref{Thm B} is a more or less immediate consequence of Theorem \ref{Thm A} and the fact that the Eulerian and the Lagrangian curvature-dimension conditions, BE$(K,N)$ and RCD$(K,N)$, are equivalent to each other
as proven in \cite{EKS-O}.

\begin{remark}
The first derivation of the transformation formula for the (Eulerian) curvature-dimension condition  BE$(K,N)$ under conformal transformation as well as under time change was presented
in  \cite{S-R} by the second  author in the setting of regular Dirichlet forms admitting a nice core of sufficiently smooth functions (``$\Gamma$-calculus in the sense of Bakry-\'Emery-Ledoux'').

Combining the techniques and results in \cite{G-N} and \cite{S-S},  the first author \cite{H-C, H-R}  proved the transformation formula for the Lagrangian    curvature-dimension condition $\rcdkn$  under conformal transformation when the reference function $w$ is bounded and smooth enough.  Together with the well-known transformation formula for $\rcdkn$ under drift transformations, this result also provides a transformation formula for $\rcdkn$ under time change.

The focus of the current paper is on proving the transformation formula for the (Eulerian or Lagrangian) curvature-dimension condition under time change in a setting of great generality (Dirichlet forms or metric measure spaces) and with minimal regularity and boundedness assumptions on $w$. 
\end{remark}

\paragraph{C.}
One of the  important applications of  time-change is the ``convexification'' of non-convex subsets $\Omega\subset X$ of an RCD$(K,N)$-space $(X,\d,\mm)$ as introdudced by the second author and Lierl  \cite{LierlSturm2018}. 
For sublevel sets  
 of regular semi-convex functions $V$,  they proved convexity  after suitable conformal transformations while control of the curvature bound under these transformations follows from the work  \cite{H-C} of the first author.  Unfortunately, these previous results do not apply to the most natural potential, the signed distance function $V=\d(.,\Omega)-\d(.,X\setminus\Omega)$ due to lack of regularity.
The more general results of the current paper, 
will apply to a suitable truncation  of the signed distance function and thus provide the following Convexification Theorem.

\begin{theorem}\label{Thm C}
Let  $\ms$ be a $\rcdkn$ space and $\Omega$ be a bounded  $\ell$-convex domain  in $(X, \d)$ with $\mm(\partial \Omega)=0$ and $\mm^+(\partial  \Omega)<\infty$.  
Then for any $N'\in (N, +\infty]$,  there exists a  Lipschitz  function  $w$  such that  the time-changed  metric measure space  $(\overline \Omega, \d^{w}, \mm^w)$ is a 
 ${\rm RCD}(K', N')$ space for some $K' \in \R$. 
\end{theorem}

 \paragraph
{Acknowledgement.}  The authors gratefully acknowledge support by the European Union through the ERC-Advanced
Grant ``Metric measure spaces and Ricci curvature -- analytic, geometric, and probabilistic challenges" (``RicciBounds").

\section{Time change and the Bakry-\'Emery condition}\label{sec:be}

This section is devoted to  study  synthetic lower Ricci bounds under time change in the setting of Dirichlet forms. More precisely, we will 
derive the transformation formula for the Bakry-\'Emery condition under time change.

\subsection{Dirichlet forms and the $\bekn$ condition}
In this part, we recall some basic facts about Dirichlet form theory and the Bakry-\'Emery theory. Firstly we make some basic assumptions on the Dirichlet form, see also \cite{Sturm96II}  
for examples satisfying these conditions.

\begin{assumption}\label{assumption1} We   assume that
\begin{itemize}
\item [a)] $(X, \tau)$ is a topological space, $(X, \mathcal{B})$  is a  measurable space and $\mm$ is a $\sigma$-finite Radon measure with full support (i.e. $\supp \mm=X$);
 ${\mathcal{B}}$ is  the $\mm$-completion of the Borel $\sigma$-algebra generated by $\tau$; and $L^p(X, \mm)$ will denote the space of  $L^p$-integrable functions on $(X, \mathcal{B}, \mm)$;
\item [b)]  $\mathcal{E}(\cdot): L^2(X, \mm) \mapsto [0,\infty]$  
is a strongly local, quasi-regular, symmetric Dirichlet form 
with domain   ${\mathbb V}:={\rm D}(\E)=\big \{ f\in L^2(X, \mm):  \E(f) < \infty \big\}$;  denote by $({P}_t)_{t>0}$ the heat semi-group generated by $\E$;
\item [c)] 
there exists an increasing sequence of (``cut-off'') functions with compact  support  $(\nchi_\ell)_{\ell\geq 1}\subset \V_\infty$ such that
 $0\leq \nchi_\ell \leq 1$, $\Gamma(\nchi_\ell)\le C$ for all $\ell$ and 
 $\nchi_\ell\to 1$, 
 $\Gamma(\nchi_\ell) \to0$ as $\ell\to\infty$, cf. \cite{Sturm95};
%
%
\item [d)]  $\E$ satisfies the Bakry-\'Emery condition $\be$ for some $K\in \R$. 
\end{itemize}
\end{assumption}
To formulate the latter, recall that 
$\V_\infty:={\rm D}(\E) \cap L^\infty(X, \mm)$ is an algebra with respect to pointwise multiplication. We say that $\mathcal{E}$ admits 
a carr\'e du champ if there exists a quadratic continuous map $\Gamma: \V\to L^1(X, \mm)$ such that
\[
\int_X\Gamma(f) \varphi\,\d\mm=\E(f, f\varphi)-\frac12\E(f^2,\varphi)\qquad\text{for all }f\in \V,\varphi\in \V_\infty.
\]
By polarization, we define $\Gamma(f, g) :=\frac14 \big ( \Gamma(f+g)-\Gamma(f-g)\big )$ and obtain $\E(f,g)=\int \Gamma(f,g)\,\d\mm$
for all $f, g\in \V$.
 It is  known that  $\Gamma$ is   local in the sense that   $\Gamma(f-g)=0$ $\mm$-a.e. on the set $\{f=g\}$.

The Dirichlet form $\E$ induces a densely defined selfadjoint operator $\Delta:{\rm D}(\Delta)\subset \V \mapsto L^2$ satisfying
$\E(f,g)=-\int g\Delta f\,\d\mm$
for all $g\in \V$. 
Put 
\begin{equation*}
 \Gamma_2(f; \varphi):=\frac12 \int \Gamma(f) \Delta \varphi\,\d\mm-\int \Gamma(f, \Delta f) \varphi\,\d\mm
\end{equation*}
and 
$
{\rm D}(\Gamma_2):=\Big\{ (f, \varphi): f,\varphi\in {\rm D}(\Delta), \
\Delta f\in\V, \
\varphi,\Delta\varphi \in {L^\infty}\Big\}
$.

%

\begin{definition}[Bakry-\'Emery condition]\label{def-bevar}
Given a function ${k} \in L^\infty$ and a  number $N\in [1,\infty]$, we say that the Dirichlet form    $\E$ satisfies the    BE$({k},N)$ condition   if it admits a carr\'e du champ and if
\begin{equation}\label{eq-bevar}
\frac12 \int \Gamma(f) \Delta \varphi\,\d\mm-\int \Gamma(f, \Delta f) \varphi\,\d\mm \geq \int \Big ({k} \Gamma(f)+ \frac1N (\Delta f)^2 \Big )\varphi\,\d\mm.
\end{equation}
for all $(f, \varphi)\in {\rm D}(\Gamma_2)$, $\varphi \geq 0$.
\end{definition}

\begin{remark} Since by our standing assumption the Dirichlet form $\E$ satisfies $\be$ for some $K\in\R$,
the ``space  of test functions''  
$$
{\rm TestF}(\E):=
\big\{f\in {\rm D}(\Delta): \ \Delta f\in \V^\infty, \, \Gamma(f)\in L^\infty\big\}$$  
is dense in $\V$   (c.f.  Section 2  \cite{AGS-B} and  Remark 2.5 therein).
Hence,  the BE$({k},N)$  condition will follow if 
\eqref{eq-bevar} holds true for all $f\in {\rm TestF}(\E)$ and all non-negative $\varphi\in {\rm D}(\Delta)\cap L^\infty$ with $\Delta\varphi \in {L^\infty}$.
\end{remark}

\begin{lemma}\label{2diff} For every $f\in{\rm D}(\Delta)$, we have $\Gamma(f)^{1/2}\in\V$ and
$$\E\Big(\Gamma(f)^{1/2}\Big)\le \int (\Delta f)^2\,d\mm- K\cdot \E(f).$$
\end{lemma}

\begin{proof} By self-improvement, the Bakry-\'Emery inequality BE$(K,\infty)$ as introduced above implies the stronger $L^1$-version 
\begin{equation*}
 \int \Gamma(f)^{1/2} \Delta \varphi\,\d\mm- \int \frac1{\Gamma(f)^{1/2}}
\Gamma(f, \Delta f) \varphi\,\d\mm \geq K\,\int \Gamma(f)^{1/2}\varphi\,\d\mm.
\end{equation*}
for all $f,\varphi\in {\rm D}(\Delta)$ with $\Delta f\in\V$, see \cite{S-S}.
Choosing $\varphi=P_t(\Gamma(f)^{1/2})$ and then letting $t\to0$ yields the claim for $f\in {\rm D}(\Delta)$ with $\Delta f\in\V$. Since the class of these $f$'s is dense in ${\rm D}(\Delta)$, the claim follows.
\end{proof}

\begin{definition} i) We say that $f \in \V^e$ if
there exists a Cauchy sequence $( f_n )_n \subset \V$   w.r.t. the semi-norm  $\E( \cdot)$ and  such that   $f_n \to f$ $\mm$-a.e. Then we 
define $\E(f):=\lmt{n}{\infty}\E(f_n)$. Similarly, $\Gamma$ can be extended to $\V^e$.

ii) We say that $f\in \V_\loc$ if for any bounded open set $U$, there is $\bar f\in \V$ such that $f=\bar f$ on $U$. Then a function $\Gamma(f)\in L^1_\loc(X, \mm)$ can be defined unambiguously by $\Gamma(f):=\Gamma(\bar f)$ on $U$.

Similarly, we define the spaces ${\rm D}_\loc(\Delta)$ and ${\rm TestF}_\loc(\E)$.
\end{definition}

\begin{definition}[Local weak Bakry-\'Emery condition]\label{def-bevar-loc}
Given a function ${k} \in L^\infty_\loc$ and a  number $N\in [1,\infty]$, we say that     the Dirichlet form $\E$ satisfies the   ${\rm BE}_\loc({k},N)$  condition if it admits a carr\'e du champ and if
\begin{equation}\label{eq-bevar-loc}
-\frac12 \int \Gamma\big(\Gamma(f), \varphi\big)\,\d\mm-\int \Gamma(f, \Delta f) \varphi\,\d\mm \geq \int \Big ({k} \Gamma(f)+ \frac1N (\Delta f)^2 \Big )\varphi\,\d\mm.
\end{equation}
for all $f\in {\rm D}_\loc(\Delta)\cap L^\infty_\loc$ with $\Delta f\in\V_\loc$
and all non-negative $\varphi\in \V^\infty$ with compact support and $\Gamma(\varphi)\in L^\infty$.
\end{definition}
Note that our standing assumption $\be$ implies that $\Gamma(f)^{1/2}\in\V_\loc$ for each $f\in{\rm D}_\loc(\Delta)$. Thus for functions $f$ and $\varphi$ as above, the term
$-\frac12 \int \Gamma\big(\Gamma(f),\varphi\big)\,\d\mm$ is well-defined.

\begin{lemma}\label{weak-strong be} $\E$ satisfies ${\rm BE}({k},N)$ for  ${k} \in L^\infty$  if and only if it satisfies ${\rm BE}_\loc({k},N)$.
\end{lemma}

\begin{proof} Assume that BE$({k},N)$ holds true and let $f$ and $\varphi$ be given as in Definition \ref{def-bevar-loc}. Choose $f'\in {\rm D}(\Delta)\cap L^\infty$ with $\Delta f'\in\V$ such that $f=f'$ on a neigborhood of $\{\varphi\not=0\}$.
Choose uniformly bounded, nonnegative $\varphi_n\in  {\rm D}(\Delta)$ with $\Gamma(\varphi_n), \Delta\varphi_n \in {L^\infty}$ such that $\varphi_n\to \varphi$ a.e.~on X and in $\V$ as $n\to\infty$. (For instance, put $\varphi_n=P_{1/n}\varphi$.) Then \eqref{eq-bevar} implies
\begin{equation*}
-\frac12 \int \Gamma\big(\Gamma(f'), \varphi_n\big)\,\d\mm-\int \Gamma(f', \Delta f') \varphi_n\,\d\mm \geq \int \Big ({k} \Gamma(f')+ \frac1N (\Delta f')^2 \Big )\varphi_n\,\d\mm>-\infty
\end{equation*}
for all $n$. Passing to the limit $n\to\infty$ yields \eqref{eq-bevar-loc}  with $f'$ in the place of $f$. Since by assumption $f=f'$ on a neigborhood of $\{\varphi\not=0\}$, this yields the claim \eqref{eq-bevar-loc}.

Conversely, assume that BE$_\loc({k},N)$ holds true and let $f$ and $\varphi$ be given as in Definition \ref{def-bevar}. 
Put $\varphi_n= P_{1/n}\varphi$ and $\varphi_{\ell,n}=\nchi_\ell\cdot P_{1/n}\varphi$ with $(\nchi_\ell)_\ell$ being the cut-off functions from assumption \ref{assumption1}. According to the BE$_\loc({k},N)$ assumption, \eqref{eq-bevar-loc} holds with  $\varphi_{\ell,n}$ in the place of $\varphi$. Passing to the limit $\ell\to\infty$ yields \eqref{eq-bevar-loc} with $\varphi_{n}$ in the place of $\varphi$ ($\forall n$).
This, however, is equivalent to \eqref{eq-bevar}, again with  $\varphi_{n}$ in the place of $\varphi$.  Finally  passing  to the limit $n\to\infty$ yields  \eqref{eq-bevar} for the given $\varphi$.
\end{proof}

\begin{remark}\label{rem-be-weak} From the proof of the preceding Lemma, it is obvious that the class of $f$'s to be considered for \eqref{eq-bevar-loc}
can equivalently be restricted to  $f\in {\rm D}_\loc(\Delta)\cap L^\infty_\loc$ with $\Delta f\in\V_\loc\cap L^\infty_\loc$.
\end{remark}

\subsection{Self-improvement of the Bakry-\'Emery condition}

The formulation of the subsequent results on the self-impovement property will require the theory of differential structures of Dirichlet forms as introduced  by Gigli in \cite{G-N}.  In order to shorten the length of the paper, we will skip the introduction of (co)tangent modules,   list  the results directly and ignore subtle differences.

\begin{proposition}[Section 2.2, \cite{G-N}]
Given  a strongly local, symmetric Dirichlet form  $\E$  admitting a carr\'e du champ $\Gamma$ defined on $\V^e$ as above. Then 
there exists  a $L^\infty$-Hilbert module  $L^2(TM)$ satisfying the following properties.
\begin{itemize}
\item [i)] $L^2(TM)$ is a Hilbert space equipped with the norm $\| \cdot \|$ such that the following   correspondence (embedding)  holds
\[
\V^e \ni f \mapsto \nabla f \in L^2(TM), ~~~~\| \nabla f\|^2=\int \Gamma(f)\,\d\mm.
\]
\item [ii)]$L^2(TM)$ is a module over the commutative ring $L^\infty(X, \mm)$.
\item [iii)]The norm $\| \cdot \|$ is  induced  by a pointwise inner product $\la \cdot, \cdot \ra$ satisfying
\[
\la \nabla f, \nabla g \ra=\Gamma(f,g)~~~~~\mm-\text{a.e.}
\]
and
\[
\la h \nabla f, \nabla g \ra=h\la  \nabla f, \nabla g \ra~~~~~\mm-\text{a.e.}
\]
for any $f, g\in \V^e$.

\item [iv)] $L^2(TM)$ is generated by $\{\nabla g: g\in \V^e\}$ in the following sense. For any $v\in L^2(TM)$, there exists a sequence $v_n=\sum_{i=1}^{M_n} a_{n,i} \nabla g_{n,i}$ with $a_{n,i}\in L^\infty$ and $g_{n,i}\in \V^e$,  such that $\| v-v_n\| \to 0$ as $n\to \infty$.
\end{itemize}
\end{proposition}

\bigskip
 By Corollary 3.3.9 \cite{G-N}, for any $f \in {\rm D}({\Delta})$ there  is  a continuous symmetric $L^\infty(M)$-bilinear map $\H_f(\cdot, \cdot)$ defined on $[L^2(TM)]^2$,   with values in $L^0(X, \mm)$.
In particular, if $f, g, h \in {\rm TestF}$ (c.f.  Lemma 3.2 \cite{S-S}, Theorem 3.3.8 \cite{G-N}),  $\H_f(\cdot, \cdot)$ is given by the following formula:
\begin{equation}\label{eq:hessian}
2\H_f(\nabla g, \nabla h)=\Gamma(g, \Gamma(f, h)) +\Gamma(h, \Gamma(f, g))-\Gamma(f, \Gamma(g,h)).
\end{equation}

\bigskip

Combining Theorem 1.4.11 and Proposition 1.4.10 in \cite{G-N}, we obtain  the following structural results. As a consequence, we can compute $\H_f(\cdot, \cdot)$ and $\Gamma(\cdot, \cdot)$ using local coordinate.

\begin{proposition}\label{decomposition}
Denote  by $L^2(TM)$ the tangent module associated with  $\E$. Then there exists  a unique  decomposition (up to $\mm$-null sets) $\{ E_n\}_{n \in \mathbb{N} \cup \{\infty\}}$ of $X$ such that
\begin{itemize}
\item [a)] For any $n \in \mathbb{N}$ and any $B \subset E_n$ with  positive measure,  $L^2(TM)$ has an orthonormal basis $\{e_{i,n}\}_{i=1}^n$ on $B$,
\item [b)] For every subset $B$ of $E_\infty$ with finite positive measure, there exists an orthonormal basis $\{e_{i,B}\}_{i \in \mathbb{N} \cup \{\infty\}} \subset L^2(TM)\restr{B}$  which generates $L^2(TM)\restr{B}$,
\end{itemize}
where we say that a countable set $\{v_i\}_i\subset L^2(TM)$ is  orthonormal  on $B$ if  $\la v_i, v_j \ra=\delta_{ij}$ $\mm$-a.e. on $B$.  By definition, the local dimension $\dim_{\rm loc}(x) \in \mathbb{N} \cup \{\infty\}$ is $n$ if $x\in E_n$.
\end{proposition}

\begin{proposition}\label{prop:finitedim}
Let $\E$ be a  Dirichlet form satisfying 
the ${\rm BE}({k},N)$ condition for some 
${k} \in L^\infty$ and some number $N\in [1,\infty]$ and let
 $\{ E_n\}_{n \in \mathbb{N} \cup \{\infty\}}$ be the decomposition given by Proposition \ref{decomposition}. 
Then  $\mm(E_n)=0$ for $n>N$, and 
%
for any $(f, \varphi) \in {\rm D}(\Gamma_2)$,  we have
\begin{eqnarray}\label{eq:prop:infinitedim}
{ \Gamma}_2(f; \varphi) &\geq& \int \Big{(} {k}\Gamma(f)+|\H_f|^2_{\rm HS}+\frac1{N-\dim_{\rm loc}}(\tr \H_f -\Delta f)^2 \Big{)}\varphi\,\d\mm 
\end{eqnarray}
 where $\frac1{N-\dim_{\rm loc}}(\tr \H_f -\Delta f)^2 $ is taken $0$ on $E_N$  by definition.
 
 The same estimate \eqref{eq:prop:infinitedim} also holds true for all 
  all $f\in {\rm D}_\loc(\Delta)\cap L^\infty_\loc$ with $\Delta f\in\V_\loc$
and all nonnegative $\varphi\in \V^\infty$ with compact support and $\Gamma(\varphi)\in L^\infty$
 provided $\E$ satisfies
the ${\rm BE}_\loc({k},N)$ condition for some 
${k} \in L^\infty_\loc$ and $N\in [1,\infty]$.
\end{proposition}

\begin{proof}
The proof for constant ${k}=K$ was given in  \cite{H-R}, Proposition 3.2 and Theorem 3.3. In fact, the proof there only relies  on  a  so-called self-improvement technique in Bakry-\'Emery theory,  which   can also be applied  to  BE$({k},N)$  case without difficulty. Also the extension via localization is straightforward.
\end{proof}

In order to proceed, we briefly recall the notion of measure-valued Laplacian ${\bf \Delta}$ as introduced in \cite{S-S, G-O}.
We say that  $f\in {\rm D}({\bf \Delta}) \subset \V^e$ if there exists a  signed Borel measure  $\mu=\mu_+-\mu_-$  charging no capacity zero sets 
such that 
 \[
\int \overline{\varphi} \,\d\mu=-\int \Gamma(\varphi.
f)\,\d \mm\]
for any $\varphi\in \V$ with quasi-continuous representative
$\overline{\varphi}\in L^1(X, |\mu|)$.
If  $\mu$ is unique,  we denote it by ${\bf \Delta} f$. If ${\bf \Delta} f \ll \mm$,  we also denote its density by $\Delta f$ if there is no ambiguity.

\begin{proposition}[See  Lemma 3.2 \cite{S-S}]\label{prop:measurebochner}
Let $\E$ be a  Dirichlet form satisfying the ${\rm BE}_\loc({k},N)$ condition.  Then for any $f\in {\rm TestF}_\loc(\E)$, we have $\Gamma(f) \in {\rm D}_\loc({\bf \Delta})$
 and 
\begin{equation*}
\frac 12 {\bf \Delta} \Gamma(f)-\Gamma(f, \Delta f)\,\mm \geq \Big{(} {k}\Gamma(f)+|\H_f|^2_{\rm HS}+\frac1{N-\dim_{\rm loc}}(\tr \H_f -\Delta f)^2 \Big{)}\,\mm.
\end{equation*}
In particular,  the singular part  of the measure  ${\bf \Delta} \Gamma(f)$ is non-negative.
\end{proposition}

\subsection{$\bekn$ condition under time change}
We define  the time-change of the  Dirichlet form $\E$ in the following way.
\begin{definition}[Time change]
Given a  function  $w\in L^\infty_\loc(X, \mm)$, define  the weighted measure $\mm^w:=e^{2w}\mm$  and  the {time-changed Dirichlet form} $\E^w$ on  $ L^2(X, \mm^w)$ by
\[
\E^w(f):=\int \Gamma(f)\,\d \mm~~~~~\forall f\in \V^w
\]
with ${\rm D}(\E^w):=\V^w:=\V^e \cap L^2(X, \mm^w)$. 
Note that indeed $\E^w(f)$ does not depend on $w$ and $(\V^w)^e=\V^e$.
\end{definition}

We leave it to the reader to verify the following simple but fundamental properties.
\begin{lemma}\label{lemma:lap}

i)  $\E^w$ is  a strongly local, symmetric Dirichlet form.  
  
ii) $\E^w$ admits a carr\'e du champ defined on  $(\V^w)_\loc=\V_\loc$  by  $\Gamma^w:=e^{-2w}\Gamma$.

iii)  Furthermore, ${\rm D}_\loc(\Delta^w)={\rm D}_\loc(\Delta)$ and $\Delta^w f=e^{-2w}\Delta f$.

iv)  If in addition $w\in \V_\loc$ then  ${\rm TestF}_\loc(\E^w)={\rm TestF}_\loc(\E)$.
\end{lemma}

Our first main result will provide the basic estimate for the Bakry-\'Emery condition under time change.


\begin{theorem}
\label{th-weakbe}
Let $w\in  {\rm D}_\loc({\bf \Delta}) \cap L^\infty_\loc $ be given and assume that $\E$ satisfies ${\rm BE}_\loc({k}, N)$ condition for some $ N\in [2, \infty)$ and ${k} \in L^\infty_\loc$.  Then for any  $N'\in (N, \infty]$,  any $f\in {\rm TestF}_\loc(\E)$ and any non-negative $\varphi \in \V_\infty$ with compact support,  we have
\begin{eqnarray}\label{eq-weakbe-1}
\lefteqn{- \int \Big [\frac12\Gamma^w\big (\Gamma^w(f),  \varphi \big)+\Gamma^w\big(f, \Delta^w f \big)\varphi\Big]\,\d\mm^w}\nonumber \\ &\qquad\qquad\geq& \int \overline{\Gamma^w(f)\varphi  }\,\d \kappa +\frac1{N'}\int (\Delta^w f)^2\varphi\,\d\mm^w
\end{eqnarray}
where  
\[
\kappa:= e^{-2w}\left ({k}-  \frac{(N-2)(N'-2)}{N'-N} \Gamma(w)\right )\mm^w- {\bf \Delta} w.
\]
\end{theorem}
\begin{proof}
By Lemma \ref{lemma:lap}  we know
\begin{eqnarray*}
&&-\frac12 \int \Gamma^w\big(\Gamma^w(f),  \varphi\big)\,\d\mm^w-\int \Gamma^w(f, \Delta^w f) \varphi\,\d\mm^w\\
& =& -\frac12 \int \Gamma\big ( e^{-2w}\Gamma(f), \varphi\big)\,\d\mm-\int \Gamma(f, e^{-2w}\Delta f) \varphi\,\d\mm\\
&=&  \Big (-\frac12 \int \Gamma\big( \Gamma(f), \varphi\big)e^{-2w}\,\d\mm+\int \Gamma(f)\Gamma(w,\varphi)e^{-2w}\,\d\mm\Big )-\Big (\int \Gamma(f, \Delta f)e^{-2w} \varphi\,\d\mm\\
&&-2\int \Delta f\Gamma(f, w) e^{-2w} \varphi\,\d\mm\Big )\\
&=& \Big( -\frac12 \int \Gamma\big( \Gamma(f), e^{-2w}\varphi\big)\,\d\mm-  \int \Gamma\big( \Gamma(f), w\big)e^{-2w} \varphi\,\d\mm\Big)+\int \Gamma(f)\Gamma(w,\varphi)e^{-2w}\,\d\mm\\
&& -\int \Gamma(f, \Delta f)e^{-2w} \varphi\,\d\mm+2\int \Delta f\Gamma(f, w) e^{-2w} \varphi\,\d\mm\\
&=& \underbrace{ \Big \{ -\frac12 \int \Gamma\big( \Gamma(f), e^{-2w}\varphi\big)\,\d\mm-\int \Gamma(f, \Delta f)e^{-2w} \varphi\,\d\mm \Big \}}_{\Gamma_2(f; e^{-2w} \varphi)}+\Big (\int \Gamma\big(w, e^{-2w}\Gamma(f)\varphi\big)\,\d\mm\\
&&-\int \Gamma\big(w, e^{-2w}\Gamma(f)\big)\varphi\,\d\mm \Big ) - \int \Gamma\big( \Gamma(f), w\big)e^{-2w} \varphi\,\d\mm+2\int \Delta f\Gamma(f, w) e^{-2w} \varphi\,\d\mm\\
&=&\Gamma_2(f; e^{-2w} \varphi)+\int \Gamma\big(w, e^{-2w}\Gamma(f)\varphi\big)\,\d\mm-\Big (\int \Gamma\big(w, \Gamma(f)\big)\varphi e^{-2w}\,\d\mm -2\int \Gamma(w)\Gamma(f)e^{-2w}\varphi \,\d\mm\Big )\\
&& - \int \Gamma\big( \Gamma(f), w\big)e^{-2w} \varphi\,\d\mm+2\int \Delta f\Gamma(f, w) e^{-2w} \varphi\,\d\mm\\
&=& (I)+(II)+(III)+(IV)-\Big [\int  \frac{(N-2)(N'-2)}{N'-N} \Gamma(w,f)^2 e^{-2w} \varphi\,\d \mm+\int \overline{ \varphi \Gamma^w(f)} \,\d  {\bf \Delta} w \Big]
\end{eqnarray*}
where
\begin{eqnarray*}
(I)&= &    \Gamma_2(f; e^{-2w} \varphi)- \int \Big (|\H_f|^2_{\rm HS} +\frac1{N-\dim_{\rm loc}}( \Delta f-\tr \H_f)^2\Big )\varphi e^{-2w}\,\d\mm,
\end{eqnarray*}
\begin{eqnarray*}
(II)&=&\int \Big{(}|\H_f|^2_{\rm HS}+2\Gamma(f)\Gamma(w)+({\dim_{\rm loc}}-2) \Gamma(f,w)^2-2\Gamma\big(w,\Gamma(f)\big)\\
&&+2\Gamma(f,w)\tr \H_f \Big{)}\varphi e^{-2w}\,\d\mm,
\end{eqnarray*}
\begin{eqnarray*}
(III)=\frac{1}{N'-\dim_{\rm loc}}\int \Big{(} \Delta f-\tr \H_f+(2-\dim_{\rm loc}) \Gamma(f,w)\Big{)}^2\varphi e^{-2w}\,\d\mm,
\end{eqnarray*}
and
\begin{eqnarray*}
(IV)&=&\int \bigg [ \Big(\frac1{N-\dim_{\rm loc}}-\frac1{N'-\dim_{\rm loc}}\Big) (\Delta f-\tr \H_f)^2\\
&&+2\Big(1-\frac{2-\dim_{\rm loc}}{N'-\dim_{\rm loc} }\Big) (\Delta f-\tr \H_f)\Gamma(f,w)\\
&&+\Big( \frac{(N-2)(N'-2)}{N'-N} -( \dim_{\rm loc}-2)-\frac{( \dim_{\rm loc}-2)^2}{N'-\dim_{\rm loc}}\Big )\Gamma(f,w)^2 \bigg ]\varphi e^{-2w}\,\d\mm.
\end{eqnarray*}
By  Proposition \ref{prop:measurebochner}  we  obtain
$$
(I) \geq  \int {k} \Gamma(f) \overline{e^{-2w}\varphi}\,\d\mm.
$$ 
By Lemma \ref{lemma:timechangehessian} below we get
$$(II)+(III) \geq \frac{1}{N'}\int \big{(} \Delta f\big{)}^2\varphi e^{-2w}\,\d\mm.
$$
As for the last term $(IV)$, it can be checked that  the function in the bracket is positive definite, so $(IV) \geq 0$.

Combining the computations above,  we complete the proof.
\end{proof}

\begin{lemma}\label{lemma:timechangehessian}
For any $f \in {\rm TestF}_\loc(\E)$, we have 
\[
A_1+\frac{1}{N'-\dim_{\rm loc}}A_2^2\geq \frac{1}{N'}\big{(} \Delta f\big{)}^2~~\mm-\text{a.e.} 
\]
where
\begin{eqnarray*}
A_1:=|\H_f|^2_{\rm HS}&+&2\Gamma(f)\Gamma(w)+({\dim_{\rm loc}}-2) \Gamma(f,w)^2\\
&-&2\Gamma(w,\Gamma(f))+2\Gamma(f,w)\tr \H_f 
\end{eqnarray*}
and
\[
A_2:= \Delta f-\tr \H_f-({\dim_{\rm loc}}-2) \Gamma(f,w).
\]
\end{lemma}

\begin{proof}
By Proposition \ref{prop:finitedim},   there exits an orthonormal basis $\{e_i\}_i \subset L^2(TM)$. Then we  denote $\H_f(e_i, e_j)$ by $(\H_f)_{ij}$ and denote $\la \nabla g, e_i \ra$ by $g_i$ for any $g\in \V^e$.  We define a matrix $H:=(H_{ij})$ by
$H_{ij}=(\H_f)_{ij}-w_if_j-w_jf_i+\Gamma(f,w)\delta_{ij}$.
Then we have
\begin{eqnarray*}
\mathop{\sum}_{i,j}H^2_{ij}
&=& \mathop{\sum}_{i,j}\big{(}(\H_f)_{ij}-w_if_j-w_jf_i+\Gamma(f,w)\delta_{ij} \big{)}^2\\
&=&|\H_f|^2_{\rm HS} +2\Gamma(f)\Gamma(w)+({\dim_{\rm loc}}-2) \Gamma(f,w)^2\\
&&~~~-2\Gamma(w,\Gamma(f))+2\Gamma(f,w)\tr \H_f \
= \ A_1
\end{eqnarray*}
 where $\tr \H_f=\mathop{\sum}_{i} (\H_f)_{ii}$.
It can also be seen that 
\[
\tr H =\mathop{\sum}_{i}H_{ii}=\tr \H_f+({\dim_{\rm loc}}-2) \Gamma(f,w)=\Delta f-A_2.
\]
Finally,  we  obtain
\begin{eqnarray*}
A_1+\frac{1}{N'-\dim_{\rm loc}}A_2^2
&=& \| H\|^2_{\rm HS}+\frac{1}{N'-\dim_{\rm loc}}\Big( \tr H-\Delta f \Big )^2\\
&\geq & \frac1{\dim_{\rm loc}}\big( \tr H \big)^2+\frac{1}{N'-\dim_{\rm loc}}\Big( \tr H-\Delta f \Big )^2\
\geq \frac1{N'} \Big( \Delta f \Big )^2
\end{eqnarray*}
which is the thesis.
\end{proof}

\begin{theorem}[BE$({k},N)$ condition under time change]\label{th:bekn}
Let  $\E$ be a Dirichlet form satisfying  the ${\rm BE}_\loc({k},N)$  condition for some  $N\in [2, \infty)$ and  ${k} \in L_\loc^\infty(X, \mm)$.  Assume that  $w\in  {\rm D}_\loc({ \bf \Delta}) \cap L_\loc^\infty$  with ${\bf \Delta} w={\bf \Delta}_{sing} w+{ \Delta}_{ac} w \, \mm$ and ${\bf \Delta}_{sing} w  \leq 0$. Moreover, assume that for some $N'\in (N, \infty]$ and $K'\in\R$   
\begin{equation}\label{eq-th:bekn}
K'\leq e^{-2w}\Big [{k}- \frac{(N-2)(N'-2)}{N'-N} \Gamma(w)- {\Delta}_{ac} w  \Big] 
\end{equation}
$\mm$-a.e.~on $X$.
Then the time-changed Dirichlet form  $\E^w$ on $L^2(X,\mm^w)$ satisfies the ${\rm BE}(K', N')$ condition.  

In particular, we have the following gradient estimate
\[
 \Gamma^w(P^w_t f) +\frac{1-e^{-2K't}}{N'K'}  (\Delta^w P_t^w f)^2\leq e^{-2K't} P^w_t \big (\Gamma^w(f)\big )~~~~~~\mm^w\text{-a.e.}
\]
for all $f\in {\rm D}(\E^w)$.
\end{theorem}

\begin{proof} Given the estimate  \eqref{eq-weakbe-1} from the previous Theorem, we iteratively will extend the class of functions for which it holds true.

i) Our first claim is that  \eqref{eq-weakbe-1} holds for all  $f\in {\rm D}_\loc(\Delta)\cap L^\infty_\loc$ with $\Delta f\in\V_\loc$
and all compactly supported, nonnegative $\varphi\in\V^\infty$ with $\Gamma(\varphi)\in L^\infty$.
Indeed, given such $f$ and $\varphi$, choose $f'\in {\rm D}(\Delta)\cap L^\infty$ with $\Delta f\in\V$ such that $f=f'$ on a neighborhood of $\{\varphi\not=0\}$.  Choose $f_n\in  {\rm TestF}(\E)$ with $f_n\to f'$ in ${\rm D}(\Delta)$ and $\Delta f_n\to \Delta f'$ in $\V$.
(For instance, put $f_n=P_{1/n}f'$.)
Applying  \eqref{eq-weakbe-1} with $f_n$ in the place of $f$ and passing to the limit $n\to\infty$ yields the claim. Indeed, 
$$\Gamma^w(\Gamma^w(f_n),\varphi)=e^{-4w}\Big[\Gamma(\Gamma(f_n),\varphi)-2\Gamma(f_n)\cdot \Gamma(w,\varphi)\Big]$$
which according to Lemma \ref{2diff} for $n\to\infty$ converges to 
$$e^{-4w}\Big[\Gamma(\Gamma(f),\varphi)-2\Gamma(f)\cdot \Gamma(w,\varphi)\Big]=\Gamma^w(\Gamma^w(f),\varphi)$$
since $f_n\to f$ in ${\rm D}(\Delta)$.

ii) Our next claim is that   \eqref{eq-weakbe-1} holds for all  $f\in {\rm D}_\loc(\Delta^w)\cap L^\infty_\loc(\mm^w)$ with $\Delta^w f\in(\V^w)_\loc^\infty$
and all compactly supported, nonnegative $\varphi\in(\V^w)^\infty$ with $\Gamma^w(\varphi)\in L^\infty(\mm^w)$. Indeed,  the conditions on $f$ and on $\varphi$ will not depend  on $w$ as long as  $w\in \V_\loc^\infty$ which is the case by assumption. This is obvious in the case of the conditions on $\varphi$.
For the conditions on $f$, note that $\Delta^w=e^{-2w}\Delta$ and $\Gamma^w(\Delta^w)\le 2e^{-4w}\big[\Gamma(\Delta f)+4(\Delta f)^2\, \Gamma(w)\big]$.

iii) Taking into account the assumptions on ${\bf \Delta} w$
and on $K'$, according to Lemma \ref{weak-strong be} together with Remark \ref{rem-be-weak} the assertion of the second claim already proves BE$(K',N')$.

iv) The gradient estimate is a standard consequence of BE$(K',N')$, see \cite{EKS-O}.
\end{proof}

\begin{remark}
Let   $e^{2w}=\rho$ and $N'=\infty$  in  Theorem \ref{th:bekn}. Then  condition \eqref{eq-th:bekn}  becomes
\[
K'\leq K\rho^{-1}+\frac12 \Delta \rho^{-1}-N\Gamma(\rho^{-\frac12})=\rho^{-\frac{N}2}\Big(  K-\frac1{N-2} \Delta \Big )\rho^{\frac{N}2-1}.
\]
  Furthermore, when $N=2$, the condition is $K'\rho \leq K-\frac12 \Delta \ln \rho$.
\end{remark}

\section{Time change and the Lott-Sturm-Villani condition
}\label{sec:rcd}

In this section, we will  study  synthetic lower Ricci bounds under time change in the setting of metric measure spaces. More precisely, we will 
derive the transformation formula for the curvature-dimension condition  of Lott-Sturm-Villani under time change.

\subsection{Metric measure  spaces and time change}
\begin{assumption}\label{assumption}
In this section we will assume that  the metric measure space $\ms$ fulfils  the following conditions:
\begin{itemize}
\item [i)] $(X,\d)$ is a complete and separable geodesic  space;
\item [ii)]$\mm$ is a $\d$-Borel measure and $\supp{\mm}=X$;
\item [iii)]  $(X,\d,\mm)$ satisfies the Riemannian curvature-dimension condition $\rcdkn$  for some $K\in \R$ and $N\in [1, \infty]$.
\end{itemize}
\end{assumption}
Given such a metric measure space $\ms$, the energy is defined on $L^2(X,\mm)$ by 
\begin{eqnarray*}
\E(f) &:=& \inf\!\Big\lbrace \liminf_{n \to \infty} \int_X \lip{f_n}^2\d\mm : f_n \in \Lip_{b}(X),\ \! f_n \to f \text{ in } L^2(X,\mm)\Big\rbrace\\
&=&\int_X\big| \D f|^2\,\d\mm
\end{eqnarray*}
where $\lip{f}(x) := \limsup_{y\to x} \vert f(x) - f(y)\vert/\!  \d(x,y)$ denotes the \emph{local Lipschitz slope}  and $|\D f|(x)$ denotes the minimal weak upper gradient at $x\in X$. We refer to \cite{AGS-C} for details.
As a part of the definition of RCD condition, $\E(\cdot)$ is a quadratic form. By polarization,  this defines a quasi-regular, strongly local, conservative Dirichlet form admitting a carr\'e du champ $\Gamma(f):=|\D f|^2$.
We use the notations $W ^{1,2}\ms=\V={\rm D}(\E)$ and   $S^2\ms=\V^e$.



\begin{definition}
Given $w\in L^2_\loc(X,\mm)$, the  time-changed  metric measure space  is defined as  $(X, \d^w, \mm^w)$ where  $\mm^w:=e^{2w}\mm$ and $\d^w$ is given  by
\begin{equation}
\d^w(x, y):=\sup \big \{\phi(x)-\phi(y): \phi \in \V_\loc \cap C(X),  |\D\phi| \leq e^w \ \mm\text{-a.e. in}~ X \big \}
\end{equation}
for any $x, y \in X$.
\end{definition}

\begin{remark}
There are various alternative definitions for the distance function under time change. The first of them is
\begin{equation}
\d_w(x, y):=\sup \big \{\phi(x)-\phi(y): \phi \in \Lip_\loc(X),  \lip\phi \leq e^w \ \mm\text{-a.e. in}~ X \big \}.
\end{equation}
Since $\Lip_\loc(X)\subset  \V_\loc \cap C(X)$ and 
$ |\D\phi|\le \lip\phi$ for $\phi\in \Lip_\loc(X)$, obviously $\d_w\le \d^w$.
It is easy to see that in both of these definitions, the class of functions under consideration can equivalently be restricted to those with compact supports. In other words,
$\d^w(x, y)=\sup \big \{\phi(x)-\phi(y): \phi \in \V \cap C_c(X),  |\D\phi| \leq e^w\ \mm\text{-a.e. in}~ X \big \}$
and
$\d_w(x, y)=\sup \big \{\phi(x)-\phi(y): \phi \in \Lip_c(X),  \lip\phi \leq e^w\ \mm\text{-a.e. in}~ X \big \}$.
\end{remark}

Moreover, we consider 
the metric $e^w\odot \d$ defined in a dual way by
\begin{equation}
\big(e^w\odot \d\big)(x,y):=\inf\Big\{\int_0^1 e^w (\gamma_s)\, |\dot\gamma_s|\,\d s:  \gamma \in {\rm AC}([0,1], X), \gamma_0=x,\gamma_1=y\Big\}.
\end{equation}
Of particular interest is 
the metric $e^{\bar w}\odot \d$
with $w$ replaced by its  upper semicontinuous envelope $\bar w$ defined by
$$\bar w(x):=\limsup_{y\to x}w(y).$$

%
%
%
%

\begin{lemma}\label{Gamma^w} Assume that $w$ is continuous a.e.~on $X$.
Then each of the metrics $\d^w, \d_w$, $e^w\odot \d$ and $e^{\bar w}\odot \d$ induces the same minimal weak upper gradient
$|\D^wf|=e^{-w}|\D f|$ $\mm$-a.e.~on $X$ for each $f\in L^2_\loc(X)$. In particular,
$$\Gamma^w=e^{-2w}\, \Gamma\quad\text{on }\V_\loc=\V_\loc^w$$
where here and henceforth  $\Gamma^w$ denotes  the carr\'e du champ operator induced by the metric measure space $(X,\d^w,\mm^w)$.
\end{lemma}

\begin{proof}
Assume that $w$ is continuous at $x\in X$. Then for each $\varepsilon>0$ there exists $\delta>0$ such that $|w(x)-w(y)|<\varepsilon$ for $y\in B_\delta(x)$.
Hence, by using appropriate truncation arguments it is easy to see that for each $\d^*\in\big\{
\d^w, \d_w, e^w\odot \d, e^{\bar w}\odot \d \big\}$
and all  $y\in B_\delta(x)$ 
$$e^{w(z)-\varepsilon}\cdot \d(x,y)\quad \le \quad
\d^*(x,y) \quad \le \quad e^{w(z)+\varepsilon}\cdot \d(x,y).$$
Hence,  $\llip^*(f)(x)=e^{-w(x)}\,\llip(x)$
for the respective local Lipschitz constants associated with $\d^*$.

To obtain the  respective minimal weak upper gradient for $f\in L^2(X,\mm)$ 
associated with $\d^*$, one has to consider the relaxations of $\llip^*(f)$ w.r.t.~the measure $\mm^w=e^{2w}\mm$. This, however, amounts to study the relaxations of the original $\llip(f)$ w.r.t.~the measure $\mm$.
Thus the claimed identify 
$\Gamma^w(f)=e^{-2w}\, \Gamma(f)$ $\mm$-a.e.~on $X$ follows.
\end{proof}

In the  following lemma we show the coincidence of $\d^w$ and $e^{\bar w}\odot \d$,  see  \cite{KSZ-ARMA} for  related results.

\begin{lemma}\label{equ-dist} Assume that $w$ is continuous a.e.~on $X$.
Then $$\d^w=e^{\bar w}\odot \d.$$
In particular,  $\d^w$ is a geodesic metric.
\end{lemma}

\begin{proof} i) Let us first prove that $\d^w$ is a geodesic metric. 
Since $X$ is locally compact w.r.t.~the metric $\d$ and since the metrics $\d^w$ and $\d$ are locally equivalent, the space $X$ is also locally compact w.r.t.~the metric $\d^w$. Therefore, it suffices to prove that  $\d^w$ is a length metric. Assume this is not the case. Then there exist points $x\not= y$ with $\d^w(x,y)<2r$ and $B^w_r(x)\cap B^w_r(y)=\emptyset$. Put 
$$f=\d^w(.,X\setminus B_r(x))-\d^w(.,X\setminus B_r(y)).$$
It is easy to verify that $\Gamma^w(f)\le 1$ and obviously $f$ is continuous.
Hence, by the very definition of $\d^w$
$$\d^w(x,y)\ge f(x)-f(y)=2r$$
which is in contradiction to our initial assumption.

ii) Now let us consider the particular case where $w$ is continuous on all of $X$.
Then  
$\d^{w}=e^{w}\odot \d$.
Indeed, both metrics are geodesic metrics on $X$ and coincide up to multiplicative pre-factors $e^{\pm\varepsilon}$ on suitable neighborhoods $B_\delta(x)$ of each point $z\in X$.

iii) To deal with the general case, let us choose a decreasing sequence of continuous functions $w_n$ with $w_n\downarrow \bar w$ as $n\to \infty$.
Then 
$\d^{w_n}=e^{w_n}\odot \d$  for each $n$  by the preceding case ii)
and thus by monotonicity for all $x,y$
$$\d^w(x,y)\le\inf_n \d^{w_n}(x,y)=\inf_n(e^{w_n}\odot \d)(x,y)=(e^{\bar w}\odot \d)(x,y).$$

iv) To prove the reverse estimate, for given $x\in X$ observe that
$f=(e^{\bar w}\odot \d)(x,.)$ is continuous and obviously 
$\llip^w(f)(y)\le 1$ in each point $y$ of continuity of $w$. Thus, in particular, $\Gamma^w(f)\le 1$ $\mm^w$-a.e.~on $X$. This indeed implies that
$\d^w(x,z)\ge |f(x)-f(z)|=(e^{\bar w}\odot \d)(x,z)$ for each $z\in X$.
\end{proof}

\begin{lemma}\label{SobLip}
Assume that $w$ is continuous a.e.~on $X$.
Then the metric measure space $(X,\d^w,\mm^w)$ has the Sobolev-to-Lipschitz property.
\end{lemma}

\begin{proof} Assume that $f\in \V_\loc$ is given with $\Gamma^w(f)\le 1$ $\mm^w$-a.e.~on $X$.
By truncation one can achieve on each bounded set $B$ that $f=f_B$ a.e.~on $B$ for some $f_B$ with bounded support and with $\Gamma^w(f_B)\le 1$ $\mm^w$-a.e.~on $X$.
Since $w\in L^\infty_\loc$, moreover, $\Gamma(f_B)\le C$ $\mm$-a.e.~on $X$. By the Sobolev-to-Lipschitz property of the original metric measure space $(X,\d,\mm)$ it follows that $f=\bar f_B$ a.e.~on $B$ for some $\bar f_B$ with $\Lip(\bar f_B)\le C$. In particular, $\bar f_B$ is continuous and $\Gamma^w(\bar f_B)\le 1$ $\mm^w$-a.e.~on $X$.
Hence, 
$$\d^w(x,y)\ge |\bar f_B(x)-\bar f_B(y)|$$ for all $x,y\in X$ by the very definition of the metric $\d^w$. In other words, $\bar f_B\in \Lip^w(X)$ with $\Lip^w(\bar f_B)\le1$.
Considering these constructions for an open covering of $X$ by such sets $B$, it follows that there exists $\bar f\in  \Lip^w(X)$ with $f=\bar f$ $\mm$-a.e.~on $X$ and $\Lip^w(\bar f)\le1$.
\end{proof}

Finally we can prove the transformation formula for the  $\rcdkn$ condition under time change.

\begin{theorem}\label{th:measurevalue}
Let  $\ms$ be a $\rcdkn$ space and let $w\in  {\rm D}_\loc({\bf \Delta})\cap  L^\infty_\loc(X)$ be continuous  $\mm$-a.e.~with ${\bf \Delta} w={\bf \Delta}_{sing} w+{ \Delta}_{ac} w \, \mm$ and ${\bf \Delta}_{sing} w  \leq 0$.  
Then the time-changed  metric measure space $(X, \d^{w}, \mm^w)$ satisfies the
 ${\rm RCD}(K', N')$ condition
 for any  $N'\in (N, +\infty]$ and $K'\in\R$ such that $\mm$-a.e.~on $X$
\begin{equation}\label{K'}
K'\le
 e^{-2w}\Big [K- \frac{(N-2)(N'-2)}{N'-N} |\D w|^2- { \Delta}_{ac} w \Big].
\end{equation}
\end{theorem}

A particular consequence of the Theorem is that the time-changed metric measure space $(X,\d^w,\mm^w)$ 
satisfies the squared exponential volume growth condition:\\ $\exists C\in\R, z\in X$:
\begin{equation}\label{vol-growth}
\mm^w(B^w_r(z))\le C \, e^{Cr^2}\qquad (\forall r>0).
\end{equation}

\begin{proof} 
From the work of \cite{AGS-B, EKS-O}, we know that the curvature dimension condition RCD$(K,N)$  implies the Bakry-\'Emery condition BE$(K,N)$ for the Dirichlet form $\E$ on $L^2(X,\mm)$ induced by the measure space $(X,\d,\mm)$. According to Theorem \ref{th:bekn}, this implies the Bakry-\'Emery condition BE$(K',N')$ for the Dirichlet form $\E^w$ on $L^2(X,\mm^w)$. Due to Lemma \ref{Gamma^w}, the latter indeed is the Dirichlet form induced by the metric measure space $(X,\d^w,\mm^w)$.
Finally, again by 
 \cite{AGS-B, EKS-O},  BE$(K',N')$ for the Dirichlet form induced by $(X,\d^w,\mm^w)$ will imply RCD$^*(K',N')$
 provided  the volume-growth condition \eqref{vol-growth} is satisfied and the Sobolev-to-Lipschitz property holds. The latter was proven in Lemma \ref{SobLip}.
 To deal with the former, we proceed in two steps.
 
 i) ) Let us  first consider the case  $w\in L^\infty(X)$. Then the volume-growth condition \eqref{vol-growth} for $(X,\d^w,\mm^w)$ obviously from that for $(X,\d,\mm)$ which in turn follows from the $\rcdkn$ assumption.
 
 ii) Now let general $w\in L_\loc^\infty(X)$ be given as well as $K'$ and $N'$ such that 
 \eqref{K'} is satisfied. Given $z\in X$, define $w_l=w\cdot \chi_\ell$ with suitable cut-off functions $(\chi_\ell)_{\ell\in\N}$ (cf.   \cite{AMS-O}, Lemma 6.7) such that for all $\ell\in\N$
 \begin{itemize}
 \item $w_\ell=w$ on $B_\ell(z)$
 \item $w_\ell$ is bounded on $X$
 \item
$ e^{-2w_\ell}\Big [K- \frac{(N-2)(N'-2)}{N'-N} |\D w_\ell|^2- { \Delta}_{ac} w_\ell \Big]\ge K'-1$.
 \end{itemize}
 Then according to part i) of this proof, the metric measure space  $(X,\d^{w_\ell},\mm^{w_\ell})$
satisfies ${\rm RCD}(K'-1,N')$. This in particular implies that there exists a constant $C$ (which indeed can be chosen independent of $\ell$) such that 
$$\mm^{w_\ell}\big(B_r^{w_\ell}(z)\big)\leq C\, e^{C\, r^2}$$
for all $r>0$. Since $\mm^{w_\ell}\big(B_r^{w_\ell}(z)\big)=\mm^{w}\big(B_r^{w}(z)\big)$ for all $r\leq\ell$, this finally proves the requested volume growth condition.
\end{proof}

It might be of certain interest to analyze the validity of the volume growth condition \eqref{vol-growth} under time change without referring to curvature bounds.
\begin{lemma}\label{prop:volumegrowth}
Suppose  there exist non-negative $p,q\in L^\infty_\loc(\R_+)$ with
$$-q(\d(\cdot, x_0)) \leq w(.)\leq p(\d(\cdot, x_0))\qquad\text{on }X.$$
Then $(X, \d^w, \mm^w)$ satisfies the squared exponential volume growth condition:\\ $\exists C\in\R, x_0\in X$:
\begin{equation}\label{vol-growth}
\mm^w(B^w_r(x_0))\le C \, e^{Cr^2}\qquad (\forall r>0)
\end{equation}
provided the function  $f(r):=\int_0^r e^{-q(s)}ds$ satisfies
\begin{itemize}
\item[(i)] $\liminf_{r\to\infty} \frac1r f(r)>0$ and
\item[(ii)] $\limsup_{r\to\infty} \frac1{r^2}p\big(f^{-1}(r)\big)<\infty$.
\end{itemize}
 In particular, if  $q$ is bounded  and $\lmts{r}{\infty}\frac{p(r)}{r^2}<\infty$,  then $(X, \d^w, \mm^w)$ satisfies the squared exponential volume growth condition.
%
\end{lemma}

\begin{proof}  From Lemma \ref{equ-dist},  we know
\begin{equation*}
\d^w(x_0, y) \geq \int_0^ {\d(x_0, y)} \exp \big( -q(r) \big )\,\d r=f(\d(x_0, y))
\end{equation*}
for any $y\in X$.  Since $f^{-1}$ is strictly increasing, this implies
\begin{equation*}
 B_{{R}}^{{w}}(x_0) \subset B_{f^{-1}{(R)}}(x_0),~~~~\forall ~R>0.
\end{equation*}
Hence,
\[
\mm^w \big (B_{R}^{w}(x_0)\big) \leq \exp \big(2p( f^{-1}(R))\big)\mm \big (B_{{f^{-1}(R)}}(x_0)\big).
\]
Recall that the $\rcd$ condition implies the squared exponential  volume growth condition,  so there exists $M, c >0$ such that 
\begin{equation*}
\mm^w \big (B_{R}^{{w}}(x_0)\big) \leq  M\exp \Big (2p( f^{-1}(R))+c \big( f^{-1}(R) \big )^2\Big).
\end{equation*}
Note that (i) implies $\limsup_{r\to\infty} \frac1r f^{-1}(r)<\infty$. Hence, together with (ii), this implies  the squared exponential volume growth condition for $(X, \d^w, \mm^w)$.
\end{proof}

\subsection{Convexity transform}
Firstly we introduce the notion of local $\ell$-convexity   in non-smooth setting. Such notion  is derived    from  \cite{LierlSturm2018}  by the second author and Lierl (see  Definition 2.6 and Definition 2.9  therein).  

\begin{definition}[$\ell$-convex functions, Definition 2.6 \cite{LierlSturm2018}]
Given $\ell\in\R$, we say that a function $V$ is $\ell$-convex on a closed subset $Z \subset  X $ if there exists a convex open covering $ \cup_i X_i \supset Z$   such that each $V\restr{X_i}: \overline {X_i} \mapsto  (-\infty, +\infty]$ is $\ell$-geodesically convex, in the sense that for each $x_0, x_1 \in \overline{X_i}$, there exists a geodesic $\gamma:[0, 1] \mapsto X$ from $x_0$ to $x_1$ such that 
\[
V(\gamma(t)) \leq (1-t)V(\gamma(0))+tV(\gamma(1))-\frac{\ell}2t(1-t)|\dot{\gamma}_t|^2,~~\forall t\in [0, 1].
\]
\end{definition}

\begin{definition}[Locally $\ell$-convex sets, Definition 2.9  \cite{LierlSturm2018}]\label{def:kconvex}
Let $\Omega \subset X$ be an open subset and let $V:= \d(., \Omega)-\d(., {X\setminus \Omega})$ denote the signed distance from the boundary, in the sequel also briefly denoted by $\pm \d(., \partial \Omega)$. 

We say that $\Omega$ is locally $\ell$-convex if  for each $\delta > 0$ there exists $r > 0$ such that   $ V$  is $(\ell-\delta)$-convex on $\Omega^r_{-r}$  with $|\D V|\geq 1-\delta$ where $\Omega^r_{-r}:= \{-r < V < r\}$.  
\end{definition}

\begin{remark}
Assume that $X$ is a smooth Riemannian manifold,  and  $\Omega$ is  a bounded open subset of $X$ with smooth boundary. It  is  proved  in Proposition 2.10  \cite{LierlSturm2018}, that the real-valued second fundamental form  on $\partial \Omega$ is bounded from below by $\ell$ if and only if  $\Omega$ is locally $\ell$-convex. 
\end{remark}

Then we can  convexify  locally $\ell$-convex sets using time change and  the following convexification technique developed in  \cite{LierlSturm2018} (see Theorem 2.17 therein).

\begin{lemma}[Convexification Theorem]
Let $\Omega$ be a  locally $\ell$-convex subset in $X$ for some $\ell\leq 0$. Then $\Omega$  is locally geodesically convex in $(X, \d^{-\ell' V}) $ for any $\ell'<\ell$.
\end{lemma}

Next we recall some important results concerning $L^1$-optimal transport and measure decomposition. This theory  has proven to be a powerful tool in studying the fine structure of  metric measure spaces. We refer the readers  to the lecture note \cite{CavallettiL1} for an overview of this topic  and the  bibliography. 

\begin{lemma}[Localization for $\rcdkn$ spaces, Theorem 3.8 and Theorem 5.1 \cite{CM-ISO})]\label{prop-l1}
Let $\ms$ be an essentially non-branching metric measure space with $\supp \mm=X$,  and satisfying $\rcdkn$ condition for some $K\in \R$ and $N \in (1, \infty)$. Then
for any 1-Lipschitz function $u$ on $X$  and the  transport set $\mathsf T_u$ associated with $u$ (up to $\mm$-measure zero set, $\mathsf T_u$ coincides with
 $\{|\nabla u|=1\}$), there is a disjoint family of unparameterized geodesics  $\{X_q\}_{q \in \mathfrak Q}$ such that 
\begin{equation}\label{eq1:l1}
\mm(\mathsf T_u  \setminus \cup X_q)=0,
\end{equation}
and 
\begin{equation}\label{eq2:l1}
\mm\restr{\mathsf T_u}=\int_{\mathfrak{Q}} \mm_q\, \d \mathfrak{q}(q),~~\mathfrak{q}(\mathfrak{Q})=1~~\text{and}~~\mm_q(X_q)=1~~\mathfrak{q}-\text{a.e.}~ q \in \mathfrak Q.
\end{equation}
Furthermore, for $\mathfrak{q}$-a.e. $q\in \mathfrak{Q}$, $\mm_q$ is a Radon measure with  $\mm_q  \ll  \mathcal H^1 \restr{X_q}$  and
$(X_q, \d, \mm_q)$ satisfies $\rcdkn$. In particular, $\mm_q=h_q \mathcal H^1 \restr{X_q}$ for some   $\cdkn$ probability  density $h_q$.
\end{lemma}

\begin{lemma}\label{prop:ir}
Let  $Y$ be a domain in $X$  with $\mm(\partial Y)=0$, and $V:= \d(., Y)-\d(., {X\setminus Y})$ be the signed distance from the boundary. Then the  transport set $\mathsf T_V$ associated with $V$ has full measure in $X$.    There is a disjoint family of unparameterized geodesics  $\{X_q\}_{q \in \mathfrak Q}$ satisfying \eqref{eq1:l1} and \eqref{eq2:l1} in Theorem \ref{prop-l1}, and a constant $r_0>0$ such that 
\[
V(a_q) \geq   0,~~~V(b_q)\leq -r_0~~~ \mathfrak{q}-\text{a.e.}~q\in \mathfrak{Q}
\]  where   $a_q=a_q(X_q), b_q=b_q(X_q)$  are  the end points of $X_q$.
\end{lemma}
\begin{proof}
Firstly, recall that $\rcdkn$ condition yields local compactness, so for any $x\in X\setminus \partial Y$, there is $z\in \partial Y$ such that $\d(x, z)=\d(x, \partial Y)$ and thus $x\in \mathsf T_V$. So  $\mm(X \setminus \mathsf T_V)=0$. 

Secondly,  by Definition \ref {def:kconvex},  $V$  is semi-convex on $Y^{r_0}_{-r_0}$ for some $r_0>0$. 
By the main theorem of \cite{S-G},  for each $x_0 \in Y^{r_0}_{-r_0}$  there exists a unique gradient flow for $V$ starting in $x_0$. In particular, there is a  maximal transport (geodesic) line $\gamma \subset \mathsf T_V$  satisfying $V(\gamma_1)-V(\gamma_0)=\d(\gamma_0, \gamma_1)$,  $V(\gamma_1) \geq V(x_0)$ and $V(\gamma_0) \leq -r_0$.

By Theorem \ref{prop-l1} there is  a disjoint family of unparameterized geodesics  $\{X_q\}_{q \in \mathfrak Q}$ such that  $\mm(\mathsf T_V \setminus \cup X_q)=0$.  In addition,    $X_q \cap \{V \leq -r_0\} \neq \emptyset$ and  $X_q \cap \{V \geq 0\} \neq \emptyset$  for any $q \in \mathfrak Q$.  Therefore $\mm(X \setminus \cup X_q)=0$,   $V(a_q) \geq   0$ and $V(b_q)\leq -r_0$.
\end{proof}

\begin{proposition}[Convexification]\label{th:convexify}
Let $\Omega$ be a  locally $\ell$-convex domain  in $(X, \d)$ for some $\ell\leq 0$, and $\mm(\partial \Omega)=0$.  Then for any $\ell'<\ell$,  there exist $r_0>0$ and a  Lipschitz  function  $w$  such that $\Omega$  is locally geodesically convex in $(X, \d^{w})$ and $w \in {\rm D}({\bf \Delta}, X \setminus \partial \Omega)$  with
\[
{\bf \Delta}  w\restr{X \setminus \partial \Omega} \leq -\ell' \Big (\cot_{K, N}( r_0/4) +\frac{2}{ r_0}\Big)\mm\restr{X \setminus \partial \Omega}.
\]
where the function $\cot_{K, N}: [0, +\infty) \mapsto [0, +\infty)$ is defined  by
\[
\cot_{K, N}(x):=\left\{\begin{array}{lll}
\sqrt{K(N-1)} \cot (\sqrt {\frac{K}{N-1}} x), &\text{if}~~ K>0,\\
({N-1})/x, &\text{if}~~K=0,\\
\sqrt{-K(N-1)} \coth(\sqrt {\frac{-K}{N-1}} x), &\text{if} ~~K<0.
\end{array}
\right.
\]
\end{proposition}
\begin{proof}
Let $V:=\pm \d(\cdot, \partial Y)$  be the signed distance from the boundary  and $r_0>0$ be the constant  in Proposition \ref{prop:ir}. Given $\ell'<\ell\leq 0$, we  can find  a smooth cut-off  function  $\phi: \R \mapsto [0, r_0]$  satisfying
 \[
\phi(t) :=\left\{\begin{array}{lll}
t, &\text{if}~~ t\in [\frac14 \ell' r_0, -\frac14 \ell'r_0]\\
-\frac 12 \ell'r_0, &\text{if}~~ t\in [ -\frac 34 \ell'r_0, +\infty)\\
\frac 12 \ell'r_0, &\text{if}~~ t\in (-\infty, \frac 34 \ell' r_0]
\end{array}
\right.
\]
with $0\leq \phi'\leq 1$, $| \phi''|\leq -\frac{2}{\ell' r_0}$ on $\R$.
Then we define $w:=\phi (-\ell'V)$. By  Convexification Theorem (c.f. Theorem 2.17 \cite{LierlSturm2018}) we know $\Omega$  is locally geodesically convex in $(X, \d^{w})$.

By chain rule (c.f. Proposition 4.11  \cite{G-O}) and  Corollary 4.16  \cite{CM-Laplacian} we have $w \in {\rm D}({\bf \Delta}, X \setminus \partial \Omega)$, and
\begin{eqnarray}\label{eq1:convexity}
{\bf \Delta} w\restr{X \setminus \partial \Omega} &=&  -\ell'\phi'(-\ell'V) {\bf \Delta} V\restr{X \setminus \partial \Omega}+(\ell')^2\phi''(-\ell'V) |\D V|^2\,\mm\restr{X \setminus \partial \Omega}.
\end{eqnarray}


 In addition,  by Corollary 4.16  \cite{CM-Laplacian} and  the fact  $ \phi''\leq -\frac{2}{\ell' r_0} $,  we obtain
\begin{eqnarray*}
&&{\bf \Delta} w\restr{X \setminus \partial  \Omega}\\
 &\leq&   -\ell'\phi'(-\ell'V) {\bf \Delta} V\restr{X \setminus  \partial \Omega}-\frac{2\ell'}{ r_0} |\D V|^2\,\mm\restr{X \setminus \partial  \Omega}\\
&\leq& -\ell'\phi'(-\ell'V) \Big (\cot_{K, N}(\d(x, b_q))\mm\restr{X \setminus \partial \Omega} +\int_{\mathfrak{Q}} h_q \delta_{b_q}\, \d \mathfrak{q}(q) \Big ) -\frac{2\ell'}{ r_0}\,\mm\restr{X \setminus \partial  \Omega}.
\end{eqnarray*}

Furthermore,   we know  $\phi'(-\ell'V)=0$ on $\{V \leq -\frac 34 r_0\} \cup \{V \geq \frac 34 r_0\}$. So  from Proposition \ref{prop:ir} we can see that  $$\phi'(-\ell'V) \int_{\mathfrak{Q}} h_q \delta_{b_q}=0.$$ Combining with the monotonicity of $\cot_{K, N}$ we obtain
\[
 {\bf \Delta}  w\restr{X \setminus \partial \Omega} \leq -\ell' \Big (\cot_{K, N}( r_0/4) +\frac{2}{ r_0}\Big)\mm\restr{X \setminus \partial \Omega}
\]
which is the thesis.
\end{proof}

Combining Lemma \ref{th:measurevalue} and Proposition \ref{th:convexify} we can prove the main theorem of this section. Recall that the Minkowski content of a 
set $Z \subset X$ is defined by
\[
\mm^+(Z):=\mathop{\liminf}_{\epsilon \to 0} \frac{\mm(Z^\epsilon)-\mm(Z)}{\epsilon}
\]
where $Z^\epsilon \subset X$ is the $\epsilon$-neighbourhood of $Z$ defined by $Z^\epsilon:=\{x: \d(x, Z)<\epsilon\}$.

\begin{theorem}\label{th:final}
Let  $\ms$ be a $\rcdkn$ space and $\Omega$ be a bounded  $\ell$-convex domain  in $(X, \d)$ with $\mm(\partial \Omega)=0$ and $\mm^+(\partial  \Omega)<\infty$.  
Then for any $N'\in (N, +\infty]$,  there exists a  Lipschitz  function  $w$  such that  $(\overline \Omega, \d^{w}, \mm^w)$ is a 
 ${\rm RCD}(K', N')$ space for some $K' \in \R$. 
\end{theorem}
\begin{proof}
Let $w$ be the reference function obtained in Proposition \ref{th:convexify}.   Denote  by  $\mu$  the trivial extension of ${\bf \Delta}  w\restr{X \setminus \partial \Omega}$ on whole $X$.  To apply Lemma \ref{th:measurevalue} and Proposition \ref{th:convexify},  it suffices to  show  that $w \in \dbdelta$ and ${\bf \Delta} w \leq \mu$.

Given an arbitrary non-negative  Lipschitz function $\varphi \in \Lip (X, \d)$ with bounded support.  For any $\epsilon>0$, there exists a Lipschitz function $\phi_\epsilon \in \Lip(\R)$ satisfying
 \[
\phi_\epsilon(t) :=\left\{\begin{array}{lll}
~~0, &\text{if}~~ t\in [0, \frac \epsilon 2]\\
~~\frac  {2} \epsilon(t-\frac \epsilon 2), &\text{if}~~ t\in [\frac \epsilon 2, \epsilon]\\
~~1, &\text{if}~~ t\in   [\epsilon, +\infty)
\end{array}
\right.
\]
Define  $\bar \varphi_\epsilon:= \phi_\epsilon(\d(\cdot, \partial \Omega)) \varphi$. By Leibniz rule and chain rule we know $\bar \varphi_\epsilon \in \Lip(X, \d)$,  and  $\supp \bar \varphi_\epsilon \subset X \setminus \partial \Omega$.
Therefore by Lemma \ref{th:convexify} and monotone convergence theorem  we get
\begin{eqnarray*}
 \int  \varphi\, \d \mu
&=& \lmt{\epsilon}{0} \int \bar \varphi_\epsilon\, \d {\bf \Delta}  w\restr{X \setminus \partial \Omega}\\
&=&-\lmt{\epsilon}{0} \int_{X \setminus \partial \Omega}  \Gamma ( \bar \varphi_\epsilon, w)\, \d \mm \\
&=& - \lmt{\epsilon}{0}\int  \phi_\epsilon(\d(\cdot, \partial \Omega)) \Gamma (\varphi, w)\, \d \mm- \lmt{\epsilon}{0}\int  \varphi \Gamma (\phi_\epsilon(\d(\cdot, \partial \Omega)), w)\, \d \mm\\
&=& - \int  \Gamma (\varphi, w)\, \d \mm- \lmt{\epsilon}{0}\int  \varphi \Gamma (\phi_\epsilon(\d(\cdot, \partial \Omega)), w)\, \d \mm.
\end{eqnarray*}

By Lemma \ref{prop-l1} we have a measure decomposition $\d \mm=\d \mm_q  \, \d\mathfrak{q}(q)$ associated with the signed distance function  $\pm \d(\cdot, \partial \Omega)$. Thus
\begin{eqnarray*}
&&\lmt{\epsilon}{0}\int  \varphi \Gamma (\phi_\epsilon(\d(\cdot, \partial \Omega)), w)\, \d \mm \\
&=& \lmt{\epsilon}{0}\int_{\Omega^{-\epsilon/ 2}_{-\epsilon} }  \varphi \Gamma (\phi_\epsilon(\d(\cdot, \partial \Omega)), w)\, \d \mm+  \lmt{\epsilon}{0}\int_{ \Omega_{ \epsilon /2}^ {\epsilon} }  \varphi \Gamma (\phi_\epsilon(\d(\cdot, \partial \Omega)), w)\, \d \mm \\
&=&\lmt{\epsilon}{0}  \frac{2\ell'} {\epsilon} \left (\int_{\Omega^{-\epsilon/ 2}_{-\epsilon} }  \varphi \, \d \mm-\int_{ \Omega_{ \epsilon /2}^ {\epsilon} }  \varphi  \, \d \mm  \right )\\
&=& \lmt{\epsilon}{0}  \frac{2\ell'} {\epsilon} \int_{\mathfrak{Q}} \Big (\int_{\Omega^{-\epsilon/ 2}_{-\epsilon}  \cap X_q}  \varphi \, \d \mm_q-\int_{ \Omega_{ \epsilon /2}^ {\epsilon}\cap X_q }  \varphi  \, \d \mm_q  \Big ) \, \d \mathfrak{q}(q).
\end{eqnarray*}
Notice  that 
\[
\mm^+(\partial \Omega)=\mathop{\liminf}_{\epsilon \to 0} \frac{\mm\big ((\partial \Omega)^\epsilon \big)}{\epsilon}=\mathop{\liminf}_{\epsilon \to 0}  \frac{1} {\epsilon} \left (\int_{\Omega^{0}_{-\epsilon} } 1 \, \d \mm+\int_{ \Omega_{ 0}^ {\epsilon} }  1  \, \d \mm  \right ).
\]
Therefore we obtain
\begin{eqnarray*}
&&\left | \int  \Gamma (\varphi, w)\, \d \mm \right |\\
&\leq & \left | \int  \varphi\, \d \mu \right |+\left |\lmt{\epsilon}{0}\int  \varphi \Gamma (\phi_\epsilon(\d(\cdot, \partial \Omega)), w)\, \d \mm\right | \\
&\leq& \max |\varphi| \left \{|\mu|(\supp \varphi)+  \lmt{\epsilon}{0}  \frac{2|\ell'|} {\epsilon}\left (\int_{\Omega^{0}_{-\epsilon} } 1 \, \d \mm+\int_{ \Omega_{ 0}^ {\epsilon} }  1  \, \d \mm  \right ) \right \}\\
&\leq&  \max |\varphi| \left \{ |\mu|(\supp \varphi)-2\ell' \mm^+(\partial \Omega)\right \}.
\end{eqnarray*}
By Riesz-Markov-Kakutani Representation theorem we know $w \in \dbdelta$.

Since $\mm_q=h_q \mathcal H^1 \restr{X_q}$ for some  $\cdkn$ probability density $h_q$,    we know $h_q$ and $(\ln h_q)'$ are bounded.  So for any $X_q$ such that  $\Omega^{- \epsilon /2}_ {-\epsilon}\cap X_q \neq \emptyset$ and  $\Omega_{ \epsilon /2}^ {\epsilon}\cap X_q \neq \emptyset$ for $\epsilon$ small enough,   we have
\begin{eqnarray*}
\left |\int_{\Omega^{-\epsilon/ 2}_{-\epsilon}  \cap X_q}  \varphi \, \d \mm_q-\int_{ \Omega_{ \epsilon /2}^ {\epsilon}\cap X_q }  \varphi  \, \d \mm_q \right | &\leq&  \Lip(\varphi h_q) \frac {3\epsilon} 2 \mathcal H^1(\Omega_{- \epsilon}^ {\epsilon}\cap X_q)=O(\epsilon^2).
\end{eqnarray*}
Hence by Lemma \ref{prop:ir} and Fatou's lemma, we obtain
\begin{eqnarray*}
&& \lmt{\epsilon}{0}  \frac{1} {\epsilon} \int_{\mathfrak{Q}} \Big (\int_{\Omega^{-\epsilon/ 2}_{-\epsilon}  \cap X_q}  \varphi \, \d \mm_q-\int_{ \Omega_{ \epsilon /2}^ {\epsilon}\cap X_q }  \varphi  \, \d \mm_q  \Big ) \, \d \mathfrak{q}(q)\\
&\geq&  \lmt{\epsilon}{0}  \frac{1} {\epsilon} \int_{q\in \mathfrak{Q}, \Omega^{r_0}_ {0}\cap X_q \neq \emptyset} \Big (\int_{\Omega^{-\epsilon/ 2}_{-\epsilon}  \cap X_q}  \varphi \, \d \mm_q-\int_{ \Omega_{ \epsilon /2}^ {\epsilon}\cap X_q }  \varphi  \, \d \mm_q  \Big ) \, \d \mathfrak{q}(q)\\
&\geq& 0.
\end{eqnarray*}
In conclusion,   we obtain  
\begin{eqnarray*}
\lefteqn{\int \varphi \, \d {\bf \Delta } w= -\int  \Gamma (\varphi, w)\, \d \mm}\\
&=&  \int  \varphi\, \d \mu+\lmt{\epsilon}{0}  \frac{2\ell'} {\epsilon} \int_{\mathfrak{Q}} \Big (\int_{\Omega^{-\epsilon/ 2}_{-\epsilon}  \cap X_q}  \varphi \, \d \mm_q-\int_{ \Omega_{ \epsilon /2}^ {\epsilon}\cap X_q }  \varphi  \, \d \mm_q  \Big ) \, \d \mathfrak{q}(q)
\leq  \int  \varphi\, \d \mu.
\end{eqnarray*}
Therefore  ${\bf \Delta} w \leq \mu$, by Lemma \ref{th:convexify} we know  ${\bf \Delta}_{sing} w \leq 0$ and  $\big( \Delta_{ac} w)^+ \in L^\infty$.  Then  by Proposition \ref{th:measurevalue} we know $(\overline \Omega, \d^{w}, \mm^w)$ is a
 ${\rm RCD}(K', N')$ space.
\end{proof}

\bibliographystyle{siam}
\bibliography{Bib}

\end{document}